\documentclass[microtype]{gtpart}
\usepackage[a4paper, total={6in, 8in}]{geometry}

\usepackage{mathtools}
\usepackage{tikz}
\usetikzlibrary{matrix,arrows,cd,calc}
\let\oldtikzcd\tikzcd
\let\oldendtikzcd\endtikzcd
\def\tikzcd{\bgroup\mathsurround0pt\oldtikzcd}
\def\endtikzcd{\oldendtikzcd\egroup}
\usepackage[capitalize]{cleveref}
\usepackage[style=numeric,maxnames=10,backend=bibtex,hyperref=true]{biblatex}
\title{Rectification of interleavings and a persistent Whitehead theorem}

\author{Edoardo Lanari}
\givenname{Edoardo}
\surname{Lanari}
\address{Institute of Mathematics, Czech Academy of Sciences, Prague, Czech Republic}

\author{Luis Scoccola}
\givenname{Luis}
\surname{Scoccola}
\address{Department of Mathematics, Northeastern University}


\arxivreference{2010.05378}

\numberwithin{equation}{section}

\newtheorem{defn}{Definition}[section]
\newtheorem{lem}[defn]{Lemma}
\newtheorem{prop}[defn]{Proposition}
\newtheorem{cor}[defn]{Corollary}

\newtheorem{conjecture}[defn]{Conjecture}

\newtheorem{thmx}{Theorem}

\theoremstyle{remark}
\newtheorem{rmk}[defn]{Remark}
\newtheorem{eg}[defn]{Example}

\newcommand{\leng}[1]{{\scriptscriptstyle #1}}
\newcommand{\interleaving}[6]{#1 : #3 \prescript{}{\leng{#6}}{\longleftrightarrow}_{\leng{#5}}\, #4 : #2}
\newcommand{\interleavingwo}[4]{#1 \prescript{}{\leng{#4}}{\longleftrightarrow}_{\leng{#3}}\, #2}
\newcommand{\define}[1]{\textbf{\boldmath{#1}}}

\newcommand{\ZZ}{\mathbf{Z}}
\newcommand{\NN}{\mathbf{N}}
\newcommand{\RR}{\mathbf{R}}
\newcommand{\II}{{\mathbf{I}}}
\newcommand{\III}{{\mathcal{I}}}
\newcommand{\JJJ}{{\mathcal{J}}}
\newcommand{\Top}{\mathbf{Top}}
\newcommand{\cTop}{\mathbf{Top}_{\mathsf{CGWH}}}
\newcommand{\sSet}{\mathbf{sSet}}

\renewcommand{\SS}{\mathbf{S}}

\newcommand{\pSm}{\SS^{\RR^m}}
\newcommand{\Sps}{\SS_{\mathsf{sc},\bullet}}
\newcommand{\Sp}{\SS_{\bullet}}
\newcommand{\Ss}{\SS_{\mathsf{sc}}}
\newcommand{\vect}{\mathbf{Vec}}
\newcommand{\SQ}{\mathbf{sq}}

\renewcommand{\R}{\mathbb{R}}

\newcommand{\N}{\mathbb{N}}
\newcommand{\NNN}{\mathcal{N}}
\newcommand{\obj}{\mathsf{Obj}}

\newcommand{\M}{\mathcal{M}}

\newcommand{\VR}{\mathsf{VR}}

\newcommand{\Met}{\mathbf{Met}}
\newcommand{\Set}{\mathbf{Set}}
\newcommand{\Grp}{\mathbf{Grp}}

\newcommand{\Ho}{\mathsf{Ho}}
\newcommand{\h}{\mathsf{h}}
\newcommand{\id}{\mathsf{id}}
\newcommand{\shi}{\mathsf{S}}

\newcommand{\floor}[1]{\lfloor #1 \rfloor}
\DeclareMathOperator{\colim}{colim}
\newcommand{\even}{\mathsf{e}}
\newcommand{\odd}{\mathsf{o}}
\renewcommand{\j}{\mathsf{j}}
\DeclareMathOperator{\Sing}{Sing}
\renewcommand{\colon}{:}

\DeclareMathOperator{\Hom}{Hom}

\newcommand{\sslash}{\mathbin{/\mkern-6mu/}}
\renewcommand{\epsilon}{\varepsilon}
\renewcommand{\phi}{\varphi}

\begin{document}

\begin{abstract}
    The homotopy interleaving distance, a distance between persistent spaces,
    was introduced by Blumberg and Lesnick
    and shown to be universal, in the sense that it is
    the largest homotopy-invariant distance for which
    sublevel-set filtrations of close-by real-valued functions are close-by.
    There are other ways of constructing homotopy-invariant distances,
    but not much is known about the relationships between these choices.
    We show that
    other natural distances differ
    from the homotopy interleaving distance in at most a multiplicative constant,
    and prove versions of the persistent
    Whitehead theorem,
    a conjecture of Blumberg and Lesnick that relates morphisms that induce interleavings
    in persistent homotopy groups to stronger homotopy-invariant notions of interleaving.
\end{abstract}

\maketitle


\section{Introduction}

\paragraph{Context.}

Many of the main theoretical tools of Topological Data Analysis (TDA) come in the form of stability theorems.
One of the best known stability theorems, due to Cohen-Steiner, Edelsbrunner, and Harer (\cite{CEH}), implies that
if \(f, g \colon X \to \R\) are sufficiently tame functions,
such as piecewise linear functions on the geometric realization of a finite simplicial complex, then 
\[
    d_B(D_n(f),D_n(g)) \leq \| f - g \|_{\infty}\,.
\]
Here, \(D_n(f)\) denotes the \(n\)-dimensional \textit{persistence diagram} of \(f\).
This consists of a multiset of points of the extended plane \(\overline{\R}^2\) that captures
the isomorphism type of the \(n\)th \textit{persistent homology} of the sublevel-sets of \(f\),
that is, of the functor \(\RR \to \vect\) obtained by composing the sublevel-set filtration 
\(r \mapsto f^{-1}(-\infty,r] \colon \RR \to \Top\) with the \(n\)th homology
functor \(H_n \colon \Top \to \vect\), where \(\RR\) denotes the poset of real numbers and \(\vect\)
denotes the category of vector spaces over some fixed field.
The distance \(d_B\) is the \textit{bottleneck distance}, a combinatorial
way of comparing persistence diagrams.

This result was later refined in \cite{CCGGO} to the \textit{algebraic stability theorem},
which says that for \(F,G \colon \RR \to \vect\) sufficiently
tame functors, one has
\[
    d_B(D(F),D(G)) \leq d_I(F,G),
\]
where, as before, \(D(F)\) denotes the persistence diagram of \(F\), which describes
the isomorphism type of \(F\), and
\(d_I\) denotes the \textit{interleaving distance}, a distance between
functors \(\RR \to C\) for any fixed category \(C\), which we recall below.

Stability theorems imply that pipelines like the following, popular in TDA,
are robust to perturbations of the input data and can be used for inference purposes:
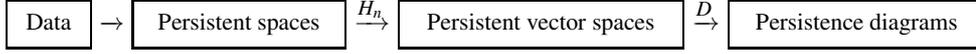
\[
\footnotesize
    \text{\framebox[1.1\width]{\begin{tabular}{c}Data\end{tabular}}} \xrightarrow{} 
    \text{\framebox[1.1\width]{\begin{tabular}{c}Persistent spaces\end{tabular}}} \xrightarrow{H_n} 
    \text{\framebox[1.1\width]{\begin{tabular}{c}Persistent vector spaces\end{tabular}}} \xrightarrow{D} 
    \text{\framebox[1.1\width]{\begin{tabular}{c}Persistence diagrams\end{tabular}}}
\]
For example, the algebraic stability theorem tells us that the last step is stable, if
we endow persistent vector spaces (\(\vect^\RR\)) with the interleaving distance and
persistence diagrams with the bottleneck distance,
while functoriality implies that the second step is stable,
if we also endow persistent spaces (\(\Top^\RR\)) with the interleaving distance (\cite{BS}).

\paragraph{Problem statement.}
Although useful in some applications,
the interleaving distance on \(\Top^\RR\) is often too fine;
for instance, it is easy to see that Vietoris--Rips and
other functors \(S \colon \Met \to \Top^\RR\) are not stable with respect to
the Gromov--Hausdorff distance on metric spaces and the interleaving distance on \(\Top^\RR\).
However, when one composes these functors with a homotopy-invariant functor,
such as homology \(H_n \colon \Top^\RR \to \vect^\RR\),
the composite \(H_n \circ S \colon \Met \to \vect^\RR\) turns out to be stable (\cite{stability-rips}).
So, in these cases, one way to make the first step in the pipeline above stable
is to force the interleaving distance on \(\Top^\RR\) to be homotopy-invariant (\cite[Section~1.2]{lb}).
For this reason, many homotopy-invariant adaptations of the interleaving
distance on \(\Top^\RR\) and related categories have been proposed (see, e.g., \cite{lesnickthesis,lb,pht}).
In order to describe some of these adaptations,
we recall the definition of the interleaving distance \(d_I\).

Let \(C\) be a category.
Given \(\delta \geq 0 \in \RR\) and \(F \colon \RR \to C\), let \(F^\delta \colon \RR \to C\)
be given by \(F^\delta(r) := F(r+\delta)\), with the obvious structure morphisms.
One says that \(F,G \in C^\RR\) are \textit{\(\delta\)-interleaved} if there exist
natural transformations \(f \colon F \to G^\delta\) and \(g \colon G \to F^\delta\)
such that \(g^\delta \circ f \colon F \to F^{2\delta}\) equals the natural
transformation \(F \to F^{2\delta}\) given by the structure morphisms of \(F\),
and such that \(f^\delta \circ g \colon G \to G^{2\delta}\) equals the natural
transformation \(G \to G^{2\delta}\) given by the structure morphisms of \(G\).
Then \(d_I(F,G) := \inf\left(\{\delta \geq 0 \,\colon\,
\text{\(F\) and \(G\) are \(\delta\)-interleaved}\} \cup \{\infty\}\right)\).

Blumberg and Lesnick (\cite{lb}) define \(X,Y\in \Top^\RR\) to be \textit{\(\delta\)-homotopy interleaved} if there exist
weakly equivalent persistent spaces \(X'\simeq X\) and \(Y' \simeq Y\) such that \(X'\) and \(Y'\) are \(\delta\)-interleaved,
and use homotopy interleavings to define the \textit{homotopy interleaving distance}, denoted \(d_{HI}\).
The homotopy interleaving distance is the (metric) quotient of the interleaving distance by the equivalence relation given by
weak equivalence, in the sense that \(d_{HI}\) is the largest homotopy-invariant distance that is bounded above
by the interleaving distance.

Instead of taking a metric quotient, one can take the categorical quotient of \(\Top^\RR\)
by weak equivalences, and define interleavings directly in the homotopy category,
similar to what is done in, e.g., \cite{lesnickthesis,pht,KS}.
In order to do this, one notes that the shift functors \((-)^\delta \colon \Top^\RR \to \Top^\RR\)
preserve weak equivalences and thus induce functors \((-)^\delta \colon \Ho\left(\Top^\RR\right) \to \Ho\left(\Top^\RR\right)\).
This lets one copy the definition of interleaving, but in the homotopy category,
which gives the notions of \textit{interleaving in the homotopy category} and of
\textit{interleaving distance in the homotopy category}, denoted \(d_{IHC}\).

A third option, also introduced in \cite{lb},
is to compare objects of \(\Top^\RR\) using interleavings in
\(\Ho\left(\Top\right)^\RR\), called \textit{homotopy commutative interleavings},
which give rise to the \textit{homotopy commutative interleaving distance}, denoted \(d_{HC}\).

We have described three homotopy-invariant notions of interleaving in decreasing order of coherence. On one end, homotopy interleavings can be equivalently described as homotopy coherent diagrams of spaces (\cite[Section~7]{lb}). On the other end, homotopy commutative interleavings correspond to diagrams in the homotopy category of spaces.
It is clear that \(d_{HI} \geq d_{IHC} \geq d_{HC}\),
and that any of the homotopy-invariant interleavings induce interleavings in homotopy groups.

Two questions arise: Are the three distances in some sense equivalent or are they fundamentally different? If a map induces interleavings in homotopy groups, does it follow that the map is part of one of the homotopy-invariant notions of interleaving? A conjectural answer to the second question is given in
\cite[Conjecture~8.6]{lb}, where it is conjectured that when \(X\) and \(Y\) are
a kind of persistent CW-complex of finite dimension \(d \in \N\),
if there exists a morphism between them inducing a \(\delta\)-interleaving
in homotopy groups, then \(X\) and \(Y\) are \(c\delta\)-homotopy interleaved,
for a constant \(c\) that only depends on \(d\).

\paragraph{Contributions.}
Homotopy interleavings compose
in any functor category of the form \(\M^{\RR^m}\) for \(\M\) a cofibrantly generated model category (\cref{triangle-inequality}).
This allows us
to state some of our results for any cofibrantly generated model category \(\M\), or
for a category of spaces \(\SS\), which can
be instantiated to be any of the Quillen equivalent model categories of topological spaces
or simplicial sets (\cref{isometry-quillen-equiv}).
Our first theorem is the following rectification result.
\begin{thmx}
    \label{strictification-step}
    Let \(\M\) be a cofibrantly generated model category, let \(X, Y \in \M^\RR\),
    and let \(\delta > 0 \in \RR\).
    If \(X\) and \(Y\) are \(\delta\)-homotopy commutative interleaved, then
    they are \(c\delta\)-homotopy interleaved for every $c > 2$.
\end{thmx}

It follows that we have \(2d_{HC} \geq d_{HI}\geq d_{IHC} \geq d_{HC}\).
The above rectification result is different from many such results in homotopy theory, where a diagram of a certain shape, in the homotopy category, is lifted to a strict diagram of the same shape.
The difference lies in the fact that the shape of the strict diagram we construct is different from the shape of the diagram in the homotopy category.
In fact, building on the suggestion in \cite{lb} of using Toda brackets to give a lower bound for the above rectification, we show (\cref{lower-bound-prop}) that for \(\M = \Top\), if \(c\,d_{HC} \geq d_{HI}\) then \(c \geq 3/2\), so that, in particular, \(d_{HC} \neq d_{HI}\).
This means that rectification in the usual sense is not possible in general, and thus standard results are not directly applicable.
We also show that \cref{strictification-step} has no analogue for multi-persistent spaces (\cref{impossibility}).

Our second theorem relates morphisms inducing interleavings in homotopy groups to interleavings in the homotopy category.
See \cref{persistent CW} for the notion of persistent CW-complex and
\cref{interleaving-in-homotopy-groups} for the notion of interleaving induced in persistent homotopy groups.

\begin{thmx}
    \label{homotopy-groups-to-noncoherent}
    Fix \(m \geq 1 \in \N\) and \(d \in \N\).
    Let \(X,Y \in \SS^{\RR^m}\) be (multi-)persistent spaces that are assumed to be
    projective cofibrant and \(d\)-skeletal if \(\SS = \sSet\), or persistent CW-complexes
    of dimension \(d\) if \(\SS = \Top\).
    Let \(\delta \geq 0 \in \RR^m\).
    If there exists a morphism in the homotopy category \(X \to Y^\delta \in \Ho\left(\SS^{\RR^m}\right)\)
    that induces \(\delta\)-interleavings in all homotopy groups,
    then \(X\) and \(Y\) are \((4(d+1)\delta)\)-interleaved in the homotopy category.
\end{thmx}

Together, \cref{strictification-step} and \cref{homotopy-groups-to-noncoherent} give
a positive answer to a version of the persistent Whitehead conjecture \cite[Conjecture~8.6]{lb} (see \cref{comparison-persistent-whitehead} for a discussion and \cref{bl-conjecture} for a statement of the conjecture).



\paragraph{Structure of the paper.}
In \cref{background-and-conventions}, we recall and give references for the necessary background.
In \cref{section-strict-interleavings}, we prove \cref{strictification-step},
we provide a lower bound for the rectification of homotopy commutative interleavings between persistent spaces,
and show that \cref{strictification-step} has no analogue for multi-persistent spaces.
In \cref{cofibrant-spaces-section}, we characterize projective cofibrant
(multi-)persistent simplicial sets as filtered simplicial sets.
In \cref{homotopy-category-and-homotopy-groups}, we prove \cref{homotopy-groups-to-noncoherent}.

\paragraph{Acknowledgements.}
The first named author gratefully acknowledges the support of Praemium Academiae of M.~Markl and RVO:67985840.
The second named author thanks Dan Christensen, Rick Jardine, Mike Lesnick, and Alex Rolle for insightful
conversations.
We thank Alex Rolle for detailed feedback and for suggesting \cref{impossibility-result} to us, Mike Lesnick for suggesting improvements to the constant of \cref{strictification-step}, and the anonymous referee for helpful feedback.

\section{Background and conventions}
\label{background-and-conventions}

The main purpose of this section is to fix notation and to provide the reader with references.
This section can be referred to as needed, but we do recommend going over \cref{interleavings-subsection} as it contains the notions of interleaving relevant to us.

We assume that the reader is comfortable with the language of category theory.
Throughout the paper, we will use the term \define{distance} to refer to any
\define{extended pseudo metric} on a (possibly large) set \(X\), that is, to any function
\(d_X : X \times X \to [0,\infty]\) that is symmetric, satisfies the triangle inequality, and is \(0\) on the diagonal.

\subsection{Spaces and model categories}

\subsubsection{Spaces}
We work model-independently whenever possible.
This means that whenever we say \define{space} we will mean either topological
space or simplicial set. Results stated for spaces will hold for both possible models.
The category of spaces will be denoted by \(\SS\).

For a general introduction to simplicial sets see, e.g., \cite{GJ} or \cite[Chapter~3]{Ho}.
We denote the geometric realization functor for simplicial sets by \(|-| \colon \sSet \to \Top\).

\subsubsection{Model categories}
The theory of model categories was introduced in \cite{Quillen}; for a modern and thorough development of this theory we recommend \cite{Ho} and \cite{H}.

We recall that two objects \(x,y\in \M\) of a model category $\M$ are said to be \define{weakly equivalent} if they are isomorphic in \(\Ho(\M) \), which happens if and only if they are connected by a zig-zag of weak equivalences in \(\M\).
This is an equivalence relation, which we denote by \(x\simeq y \).
When there is risk of confusion, morphisms in \(\Ho(\M)\) will be surrounded by square brackets \([f]\), to distinguish them from morphisms in \(\M\).

Two of the main model structures of interest to us are the
\define{Quillen model structure} on \(\Top\), the category of topological spaces
(\cite[Chapter~1,~Section~2.4]{Ho}), and the
\define{Kan--Quillen model structure} on \(\sSet\), the category of simplicial sets (\cite[Chapter~3]{Ho}).
We recall that the 
geometric realization functor \(\vert - \vert \colon \sSet \to \Top\) is left adjoint to the
singular functor \(\Sing \colon \Top \to \sSet\), and that, together, they form a \define{Quillen equivalence} (\cite[Chapter~1,~Section~1.3]{Ho}, \cite[Theorem~3.6.7]{Ho}).
For completeness, we mention that there is a subcategory \(\cTop \subseteq \Top\), the category of
compactly-generated weakly Hausdorff topological spaces (called \textit{compactly generated spaces} in \cite[Definition~2.4.21]{Ho}),
that is often used instead of \(\Top\).
The Quillen model structure on \(\Top\) restricts to a model structure on \(\cTop\), and the inclusion \(\cTop \to \Top\)
is part of a Quillen equivalence (\cite[Theorem~2.4.25]{Ho}).
This model structure is, in some respects, better behaved than
the Quillen model structure on topological spaces, and is in fact the model of space used in \cite{lb}.
We will not concern ourselves with these subtleties
since, by the observations in \cref{isometry-quillen-equiv}, there is no essential
difference between using \(\Top\) or \(\cTop\) when studying homotopy-invariant notions of interleaving.

We will make use of the notion of \define{cofibrantly generated model category} (\cite[Chapter~2,~Section~2.1]{Ho}).
Recall that the Kan--Quillen model structure on simplicial sets is cofibrantly generated,
where a set of generating cofibrations consists of the boundary inclusions
\(\partial \Delta^n \hookrightarrow \Delta^n\) for \(n\geq 0 \) (\cite[Theorem~3.6.5]{Ho}).
The Quillen model structure on topological spaces is also cofibrantly generated, with
a set of generating cofibrations given by \(\{S^{n-1}\hookrightarrow D^n\}_{n \geq 0}\) (\cite[Theorem~2.4.19]{Ho}).

We conclude by recalling the basic properties of projective model structures.
Given a model category \(\M\) and a small category \(\II\), the \define{projective model structure} on the functor category \(\M^{\II}\) is, when it exists,
the model structure whose fibrations (respectively weak equivalences) are those which are 
pointwise fibrations (respectively weak equivalences) of \(\M\).

The projective model structure on \(\M^\II\) exists, and is cofibrantly generated, whenever \(\M\) is cofibrantly generated.
Moreover, if \(\III\) and \(\JJJ\) are, respectively, generating cofibrations and generating trivial cofibration for the model
structure of \(\M\), then
\(\left\{\II(i,-)\odot f \ \colon \ i\in \II,\, f\in \III \right\}\) and
\(\left\{\II(i,-)\odot g \ \colon \ i\in \II,\, g\in \JJJ \right\}\) are, respectively, generating cofibrations and generating trivial
cofibrations for the projective model structure, where, given a
functor \(F\colon \II \to \Set\) and an object \(X \in \M\),
the functor \(F\odot X \colon \II\to \M\) is defined by \(i\mapsto \coprod_{a\in F(i)} X\)
(\cite[Section~11.6]{H}).
For simplicity, we denote \(\II(i,-) \odot X\) by \(i \odot X\).

We are especially interested in the projective model structure when the indexing category is a poset \((P,\leq)\).
In this case, if \(r \in P\) and \(X \in \M\), then \(r\odot X\)
is the functor that takes the value \(X\) on every \(s \geq r\), and
has as value the initial object of \(\M\) when \(s \nleq r\).
The non-trivial structure morphisms of this functor are the identity of \(X\).

Note that we have a functor \(\h \colon \Ho\left(\M^\II\right) \to \Ho(\M)^\II\)
by the universal property of \(\Ho\left(\M^\II\right)\).

\subsection{Interleavings and interleavings up to homotopy}
\label{interleavings-subsection}

\subsubsection{Strict interleavings}
\label{strict-interleavings-section}

We denote the poset of real numbers with their standard order by \(\RR\), and
for \(m \in \N\), we let \(\RR^m\) be the set of \(m\)-tuples of real numbers
with the product order. We set \(\overline{m} = \{i \colon 1 \leq i \leq m\}\), so that \((\epsilon_i)_{i \in \overline{m}} \leq (\delta_i)_{i\in \overline{m}} \in \RR^m\)
if and only if \(\epsilon_i \leq \delta_i\) for all \(1 \leq i \leq m\). We denote the element \((0, \dots, 0) \in \RR^m\) by \(0\).

Fix a category \(C\) and a natural number \(m \geq 1\).
An \define{\(m\)-persistent object} of \(C\) is any functor of the form \(\RR^m \to C\).
We often refer to \(m\)-persistent objects simply as \define{persistent objects}
or as \define{multi-persistent objects} when we want to stress the fact that \(m\)
is not necessarily \(1\).
Fix persistent objects \(X,Y,Z \in C^{\RR^m}\), \(r,s \in \RR^m\), and \(\epsilon,\delta \geq 0 \in \RR^m\).
We use the following conventions.

\begin{itemize}
    \item For \(f \colon X \to Y\) a natural transformation, denote the \(r\)-component of \(f\) by \(f_r \colon X(r) \to Y(r)\).

    \item Assume \(r \leq s\). The structure morphism \(X(r) \to X(s)\) will be denoted by \(\phi^X_{r,s}\).

    \item The \define{\(\delta\)-shift to the left} of \(X\) is the functor \(X^\delta \colon \RR^m \to C\) defined by \(X^\delta(r) = X(r+\delta)\),
        with structure morphisms \(\phi^{X^\delta}_{r,s} := \phi^{X}_{r+\delta,s+\delta}\).
        Shifting to the left gives a functor \((-)^\delta : C^{\RR^m} \to C^{\RR^m}\).
        Dually, there is a \define{\(\delta\)-shift to the right} functor \(\delta \cdot (-) : C^{\RR^m} \to C^{\RR^m}\)
        defined by mapping \(X\) to the persistent object \(\delta \cdot X\), with values given by \((\delta \cdot X)(r) = X(r-\delta)\).

    \item Natural transformations \(f \colon X \to Y^\delta\) will be referred to as
        \define{\(\delta\)-morphisms}, and will often be denoted by \(f \colon X \to_\delta Y\).
        Since we have natural bijections
        \(\Hom(\epsilon \cdot X,Y^\delta) \cong \Hom(X, Y^{\epsilon + \delta}) \cong \Hom((\epsilon + \delta) \cdot X, Y)\),
        we can treat a \(\delta\)-morphism \(f \colon X \to_\delta Y\) as \(f \colon X \to Y^\delta\) or as \(f \colon \delta \cdot X \to Y\).

    \item Assume \(\epsilon \leq \delta\) and let \(f \colon X \to_\epsilon Y\).
        We can compose the \(r\)-component of
        \(f\) with \(\phi^Y_{r+\epsilon,r+\delta} \colon Y(r+\epsilon) \to Y(r+\delta)\),
        giving \(\phi^Y_{r+\epsilon,r+\delta} \circ f_r \colon X(r) \to Y(r+\delta)\).
        Together, these components define the \define{shift} from \(\epsilon\) to \(\delta\)
        of \(f\), which is a \(\delta\)-morphism denoted \(\shi_{\epsilon,\delta}(f) \colon X \to_{\delta} Y\).

    \item Note that an \(\epsilon\)-morphism \(f \colon X \to_\epsilon Y\) can be composed with
        a \(\delta\)-morphism \(g \colon Y \to_\delta Z\), yielding an \((\epsilon+\delta)\)-morphism \(g^\epsilon \circ f \colon X \to_{\epsilon+\delta} Y\).
        This composition is associative and unital, and is natural with respect to shifts of morphisms.

    \item An \define{\((\epsilon,\delta)\)-interleaving} between \(X\) and \(Y\) consists of
        an \(\epsilon\)-morphism \(f \colon X \to_\epsilon Y\) together with a \(\delta\)-morphism \(g \colon Y \to_\delta X\)
        such that \(g^\epsilon \circ f = \shi_{0,\epsilon+\delta}(\id_X)\) and \(f^\delta \circ g = \shi_{0,\epsilon+\delta}(\id_Y)\).
        By \define{\(\delta\)-interleaving} we mean a \((\delta,\delta)\)-interleaving.

    \item If \(f \colon X \to_\epsilon Y\) and \(g \colon Y \to_\delta X\) form an \((\epsilon,\delta)\)-interleaving, we write \(\interleaving{f}{g}{X}{Y}{\epsilon}{\delta}\).
\end{itemize}

    Let \(\epsilon_1,\epsilon_2,\delta_1,\delta_2\geq 0 \in \RR^m\).
        Note that an \((\epsilon_1,\epsilon_2)\)-interleaving between \(X\) and \(Y\) can be composed
        with any \((\delta_1,\delta_2)\)-interleaving between \(Y\) and \(Z\), yielding an \((\epsilon_1+\delta_1,\epsilon_2+\delta_2)\)-interleaving.
The fact that interleavings compose implies that, when \(m = 1\), the formula
\[
    d_I(X,Y) = \inf\left(\{\delta \geq 0 \in \RR \colon \text{\(X\) and \(Y\) are \(\delta\)-interleaved}\}\cup \{\infty\}\right)
\]
defines an extended pseudo metric \(d_I \colon \obj\left(C^\RR\right) \times \obj\left(C^\RR\right) \to [0,\infty]\).
This is the \define{interleaving distance} on the class of objects of the category \(C^\RR\).
This notion of distance can be extended to objects of the functor category \(C^{\RR^m}\) (\cite{lesnickthesis}), but we will not make use of this extension.

\subsubsection{Interleavings up to homotopy}

If one is comparing objects of a category of functors of the form \(\RR^m \to \M\), for \(\M\) a model category,
it makes sense to want to find a homotopy-invariant notion of interleaving.
In this paper, we consider the following three homotopy-invariant relaxations of the notion of interleaving.
Let \(\M\) be a cofibrantly generated model category
and endow \(\M^{\RR^m}\) with the projective model structure.
Let \(X,Y \in \M^{\RR^m}\) and let \(\epsilon, \delta \geq 0 \in \RR^m\).

\begin{enumerate}
\item
Following \cite{lb}, we say that \(X\) and \(Y\) are \define{\((\epsilon,\delta)\)-homotopy interleaved} if there
exist \(X\simeq X'\) and \(Y\simeq Y'\) such that \(X'\) and \(Y'\) are \((\epsilon,\delta)\)-interleaved.

\item
Note that the shift functor \((-)^\delta \colon \M^{\RR^m} \to \M^{\RR^m}\)
maps weak equivalences to weak equivalences.
This implies that all the notions in \cref{strict-interleavings-section} have analogues in
the category \(\Ho\left(\M^{\RR^m}\right)\).
We say that \(X\) and \(Y\) are \define{\((\epsilon,\delta)\)-interleaved in the homotopy category}
if they are \((\epsilon,\delta)\)-interleaved as objects of \(\Ho\left(\M^{\RR^m}\right)\).

\item
Finally, as also done in \cite{lb},
we say that \(X\) and \(Y\) are \define{\((\epsilon,\delta)\)-homotopy commutative interleaved}
if their images \(\h X,\h Y \colon \RR^m \to \Ho(\M)\) are \((\epsilon,\delta)\)-interleaved.
\end{enumerate}

An \((\epsilon,\delta)\)-homotopy interleaving gives rise to an \((\epsilon,\delta)\)-interleaving
in the homotopy category, which, in turn, gives rise to an \((\epsilon,\delta)\)-homotopy commutative interleaving.

For each of the three homotopy-invariant notions of interleaving introduced above,
we have a corresponding extended pseudo metric
on the collection of objects of the category \(\M^\RR\).
Let \(X,Y \in \M^\RR\).
Following \cite{lb}, we define the \define{homotopy interleaving distance} as
\[
    d_{HI}(X,Y) = \inf\left(\{\delta \geq 0 \in \RR \colon \text{\(X\) and \(Y\) are \(\delta\)-homotopy interleaved}\}\cup \{\infty\}\right).
\]
The fact that the homotopy interleaving distance satisfies the triangle inequality follows from \cref{triangle-inequality}.
The \define{interleaving distance in the homotopy category} is
\[
    d_{IHC}(X,Y) = \inf\left(\{\delta \geq 0 \in \RR \colon \text{\(X\) and \(Y\) are \(\delta\)-interleaved in the homotopy category}\}\cup \{\infty\}\right).
\]
Again following \cite{lb}, the \define{homotopy commutative interleaving distance} is defined as
\[
    d_{HC}(X,Y) = \inf\left(\{\delta \geq 0 \in \RR \colon \text{\(X\) and \(Y\) are \(\delta\)-homotopy commutative interleaved}\}\cup \{\infty\}\right).
\]

\begin{rmk}
    \label{isometry-quillen-equiv}
Note that if \(\M \rightleftarrows \NNN\) is a Quillen equivalence between cofibrantly generated model categories, then the induced Quillen equivalence (\cite[Theorem~11.6.5]{H}) \(\M^{\RR^m} \rightleftarrows \NNN^{\RR^m}\) between the projective model structures respects
interleavings, in the sense that shifts commute with both the left and right adjoints.
This implies that, for any of the three homotopy-invariant notions of interleaving described above, we have that two functors on one side of the adjunction are \((\epsilon,\delta)\)-interleaved if and only if their images (along the derived adjunction) on the other side are \((\epsilon,\delta)\)-interleaved.
In particular, if \(m=1\), the two adjoints give an isometry between \(\M^\RR\) and \(\NNN^\RR\) independently of whether we use \(d_{HI}\), \(d_{IHC}\), or \(d_{HC}\).
\end{rmk}

\subsubsection{Composability of homotopy interleavings}
In this short section, we give a simplified proof of a generalization of the fact that homotopy interleavings
can be composed, originally proved in \cite[Section~4]{lb}.
This is generalized further in \cite[Theorem~4.1.4]{s}.

\begin{lem}
    \label{pullback-of-interleaving}
    Let \(C\) admit pullbacks. Fix \(m \geq 1 \in \N\), objects \(X, Y, B \colon \RR^m \to C\),
    elements \(\epsilon,\delta \geq 0 \in \RR^m\), an \((\epsilon,\delta)\)-interleaving
    \(\interleaving{f}{g}{X}{Y}{\epsilon}{\delta}\), and a map \(h \colon B \to Y\).
    The pullback of \(f \colon X \to Y^\epsilon\) along \(h^\epsilon \colon B^\epsilon \to Y^\epsilon\),
    denoted \(k \colon A \to B^\epsilon\), is part of an \((\epsilon,\delta)\)-interleaving
    \(\interleaving{k}{l}{A}{B}{\epsilon}{\delta}\).
\end{lem}
\begin{proof}
    We start by depicting the pullback square in the statement:
\[
    \begin{tikzpicture}
      \matrix (m) [matrix of math nodes,row sep=2em,column sep=2em,minimum width=2em,nodes={text height=1.75ex,text depth=0.25ex}]
        { A & B^\epsilon\\
          X & Y^\epsilon. \\};
        \path[line width=0.75pt, -{>[width=8pt]}]
        (m-1-1) edge node [above] {$k$} (m-1-2)
                edge node [left] {} (m-2-1)
        (m-2-1) edge node [above] {$f$} (m-2-2)
        (m-1-2) edge node [right] {$h^\epsilon$} (m-2-2)
        ;
    \end{tikzpicture}
\]
    Consider the morphisms \(i = \shi_{0,\epsilon+\delta}(\id_B) \colon \delta \cdot B \to B^\epsilon\) and
    \(g\circ (\delta \cdot h) \colon \delta \cdot B \to X\). Since \(f \circ g \circ (\delta \cdot h) = h^\epsilon \circ i\),
    the universal property of \(A\) gives us a map \(l \colon \delta \cdot B \to A\), or equivalently,
    a map \(l \colon B \to A^\delta\).
    By construction, \(k^\delta \circ l = \shi_{0,\epsilon+\delta}(\id_B) \colon B \to B^{\epsilon+\delta}\).
    To prove that \(l^\epsilon \circ k = \shi_{0,\epsilon+\delta}(\id_A) \colon A \to A^{\epsilon+\delta}\),
    or equivalently that \(l^\epsilon \circ k = \shi_{0,\epsilon+\delta}(\id_A) \colon \epsilon \cdot A \to A^\delta\),
    apply the functor
    \((-)^\delta \colon C^{\RR^m} \to C^{\RR^m}\) to the pullback square above, and use the uniqueness part of its universal property.
\end{proof}

\begin{prop}[cf. {\cite[Section~4]{lb}}]
    \label{triangle-inequality}
    Let \(\M\) be cofibrantly generated, fix \(m \geq 1\), let \(X,Y,Z \colon \RR^m \to \M\), and let \(\epsilon_1,\epsilon_2,\delta_1,\delta_2 \geq 0 \in \RR^m\).
    If \(X\) and \(Y\) are \((\epsilon_1,\epsilon_2)\)-homotopy interleaved and \(Y\) and \(Z\) are \((\delta_1,\delta_2)\)-homotopy interleaved,
    then \(X\) and \(Z\) are \((\epsilon_1+\delta_1,\epsilon_2+\delta_2)\)-homotopy interleaved.
\end{prop}

\begin{proof}
    Given interleavings \(\interleavingwo{X'}{Y'}{\epsilon_1}{\epsilon_2}\)
    and \(\interleavingwo{Y''}{Z'}{\delta_1}{\delta_2}\) with \(X \simeq X'\), \(Y' \simeq Y \simeq Y'\), and
    \(Z' \simeq Z\), we must construct an interleaving \(\interleavingwo{X''}{Z''}{\epsilon_1+\delta_1}{\epsilon_2+\delta_2}\)
    with \(X'' \simeq X\) and \(Z'' \simeq Z\).

    Since \(\M\) is cofibrantly generated, the projective model structure on \(\M^{\RR^m}\) exists, and, by applying a
    functorial fibrant replacement \(\M \to \M\) pointwise, we get a functorial fibrant replacement \(\M^{\RR^m} \to \M^{\RR^m}\).
    By construction, the fibrant replacement \(\M^{\RR^m} \to \M^{\RR^m}\) commutes with \((-)^\delta \colon \M^{\RR^m} \to \M^{\RR^m}\)
    so, in particular, it preserves interleavings.
    With this in mind, we can assume that \(Y'\) and \(Y''\) are fibrant,
    which implies, and this is a general fact, that we have \(C \in \M^{\RR^m}\) and trivial fibrations \(C \to Y'\) and \(C \to Y''\).
    Using \cref{pullback-of-interleaving}, we can pull back the interleavings we were given along the trivial fibrations,
    as follows:
    \[
    \begin{tikzpicture}
      \matrix (m) [matrix of math nodes,row sep=1em,column sep=3.15em,minimum width=1em,nodes={text height=1.75ex,text depth=0.25ex}]
        { & X'' & C & Z'' & \\
          X' & Y' & & Y'' & Z'. \\};
      \path[line width=0.75pt, -{>[width=8pt]}]
        (m-1-3) edge node [right] {} (m-2-2)
        (m-1-3) edge node [left] {} (m-2-4)
        (m-1-2) edge [dotted] node [above] {} (m-2-1)
        (m-1-4) edge [dotted] node [above] {} (m-2-5)
        (m-1-2) edge [dotted] node [below] {} node [at end, above] {$\leng{\epsilon_1}$}(m-1-3)
        (m-1-3) edge [dotted] node [above] {} node [at end, above] {$\leng{\epsilon_2}$}(m-1-2)
        (m-2-1) edge node [below] {} node [at end, below] {$\leng{\epsilon_1}$}(m-2-2)
        (m-2-2) edge node [above] {} node [at end, below] {$\leng{\epsilon_2}$}(m-2-1)
        (m-1-3) edge [dotted] node [below] {} node [at end, above] {$\leng{\delta_1}$}(m-1-4)
        (m-1-4) edge [dotted] node [above] {} node [at end, above] {$\leng{\delta_2}$}(m-1-3)
        (m-2-4) edge node [below] {} node [at end, below] {$\leng{\delta_1}$}(m-2-5)
        (m-2-5) edge node [above] {} node [at end, below] {$\leng{\delta_2}$}(m-2-4)
        ;
    \end{tikzpicture}
    \]
    Since trivial fibrations are stable under pullback, we have that \(X'' \simeq X\) and \(Z'' \simeq Z\),
    and since interleavings compose, we have that \(X''\) and \(Z''\) are \((\epsilon_1+\delta_1,\epsilon_2+\delta_2)\)-interleaved, as required.
\end{proof}

We remark that the idea of using pullbacks to prove a triangle inequality appears in \cite{Me}.

\section{Interleavings in \(\M^\RR\) and in \(\Ho(\M)^\RR\)}
\label{section-strict-interleavings}

This section is concerned with the rectification of homotopy commutative
interleavings into homotopy interleavings.
In \cref{upper-bound-section}, we prove \cref{strictification-step},
which allows one to construct, for any $c > 2$, a \(c\delta\)-homotopy
interleaving out of a \(\delta\)-homotopy commutative interleaving, when working with \(1\)-persistent objects
of any cofibrantly generated model category \(\M\).
We think of this result as giving a multiplicative upper bound of \(2\) for this rectification.
In \cref{lower-bound-section}, we give a multiplicative lower bound of \(3/2\) for the rectification, 
when \(\M\) is the category of spaces.
In \cref{impossibility}, we show that \cref{strictification-step} has no analogue for multi-persistent spaces.

\subsection{Upper bound}
\label{upper-bound-section}

Let \(\ZZ \subseteq \RR\) denote the posets of integers and real numbers respectively.
The inclusion \(i \colon \ZZ \to \RR\) induces a restriction functor \(i^* \colon C^\RR \to C^\ZZ\) for any category \(C\).
Given \(A \colon \ZZ \to C\), let \(i_*(A) \colon \RR \to C\) be given by \(A\) precomposed with
the functor \(\floor{-} \colon \RR \to \ZZ\), where \(\floor{r}\) is the largest integer bounded above by \(r\).
Note that, given \(m \geq 0 \in \ZZ\), one has a notion of \(m\)-interleaving
between functors \(A, B \colon \ZZ \to C\), and that \(i_* \colon C^\ZZ \to C^\RR\) preserves these interleavings.

We start with a few simplifications. Given \(\delta > 0\), let \(M_\delta \colon \RR \to \RR\) be given by \(M_\delta(r) = \delta \times r\).
The following lemma allows us to work with integer-valued interleavings instead of \(\delta\)-interleavings, and its
proof is immediate.

\begin{lem}
    \label{delta-to-one}
    Let \(\delta > 0 \in \RR\) and let \(m \geq 1 \in \ZZ\).
    Then \(X,Y \in C^\RR\) are \(\delta\)-interleaved if and only if
    \((M_{\delta/m})^*(X)\) and \((M_{\delta/m})^*(Y)\) are \(m\)-interleaved.\qed
\end{lem}

The following lemma allows us to work with \(\ZZ\)-indexed persistent objects instead of \(\RR\)-indexed ones.
Here, by homotopy interleaving between \(\ZZ\)-indexed functors we mean the obvious adaptation of the notion
of homotopy interleaving to \(\ZZ\)-indexed functors with values in a model category.

\begin{lem}
    \label{from-R-to-Z}
    Let \(\M\) be cofibrantly generated.
    Let \(X,Y \in \M^\RR\) and let \(m \geq 1 \in \ZZ\). If \(i^*(X),i^*(Y) \in \M^{\ZZ}\) are \(m\)-homotopy interleaved,
    then \(X\) and \(Y\) are \((m+2)\)-homotopy interleaved.
\end{lem}

\begin{proof}
    Note that \(X\) is \(1\)-interleaved with \(i_*(i^*(X))\), as, 
    for all \(r \in \RR\), we have \(r - 1 \leq \floor{r} \leq r \leq \floor{r} + 1\).
    Since \(i_*\) preserves interleavings and weak equivalences,
    it is enough to show that homotopy interleavings between \(\ZZ\)-indexed functors with values in a cofibrantly generated model category compose, which is a straightforward adaptation of \cref{triangle-inequality} to \(\ZZ\)-indexed functors.
\end{proof}

The next straightforward lemma gives us a special replacement of an object of the category \(\M^\ZZ\), with
\(\M\) a model category, that will be useful when lifting structure from \(\Ho(\M)^\ZZ\)
to \(\M^\ZZ\). 

\begin{lem}
\label{l:nice repl}
Given a model category \(\M\) and \(X \in \M^{\ZZ}\), there exists \(\overline{X} \in \M^{\ZZ}\) and a weak equivalence \(\overline{X}\to X\), such that the following properties are satisfied:
\begin{itemize}
    \item \(\overline{X}(i)\) is cofibrant in \(\M\) for every \(i \in \NN\);
    \item for every \(i\geq 0\), the structure morphism \(f_i \colon \overline{X}(i) \to \overline{X}(i+1)\) is a cofibration in \(\M\).
\end{itemize}
Dually, we can replace \(Y\in\M^{\ZZ}\) by a pointwise fibrant $\overline{Y}$ whose ``negative'' maps are fibrations.
\end{lem}

The following lemma will allow us to lift interleavings
in \(\Ho(\M)^\ZZ\) to homotopy interleavings in \(\M^\ZZ\).

\begin{lem}
    \label{h-is-smothering}
    Let \(\M\) be a model category.
    The functor \(\h : \Ho(\M^\ZZ) \to \Ho(\M)^\ZZ\) is essentially surjective, conservative, and full.
    In particular, if \(A,B \in \M^\ZZ\) become isomorphic in \(\Ho(\M)^\ZZ\), then they are weakly equivalent.
\end{lem}
\begin{proof}
It is clear that the functor is essentially surjective and full, so we only prove the last property.
Assume given \(X,Y \in \Ho(\M^{\ZZ})\) together with a map \(f\colon \h X \to \h Y\). Thanks to \cref{l:nice repl}, we can assume that \(X\) (respectively \(Y\)) is pointwise cofibrant (respectively fibrant) in \(\M\), and that all the non-negative (respectively negative) structural maps in \(X\) (respectively \(Y\)) are cofibrations (respectively fibrations). The map \(f\) can therefore be represented as a family \(\{[f_i]\}_{i\in\ZZ}\) of homotopy classes of maps of \(\M\). We construct a preimage of \(f\) under \(\h\) inductively, starting with a choice of representatives \(f'_i\) for the homotopy classes \([f_i]\). The squares
\[\begin{tikzcd}
X(-1) \ar[d,"f'_{-1}"] \ar[r,"x_{-1}"] & X(0) \ar[d,"f'_{0}"] \ar[r,"x_{0}"] & X(1) \ar[d,"f'_1"]\\
 Y(-1)  \ar[r,"y_{-1}"] & Y(0)  \ar[r,"y_{0}"] & Y(1)  
\end{tikzcd}\]
commute up to homotopy, and since \(x_0\) and \(y_{-1}\) are, respectively, a cofibration and a fibration, we can deform \(f'_1\) and \(f'_{-1}\) into homotopic maps \(f_1\colon X_1\to Y_1\) and \(f_{-1}\colon X_{-1}\to Y_{-1}\), which render the above squares commutative. Inductively, we can iterate this procedure to find the desired preimage of \(f\) under \(\h\).
\end{proof}

The next result is the main rectification step involved in lifting
interleavings in \(\Ho(\M)^\ZZ\) to homotopy interleavings in \(\M^\ZZ\).

\begin{prop}
    \label{noncoherent-to-strict-zed}
    Let \(\M\) be a model category and let \(A,B \in \M^\ZZ\).
    Let \(m \geq 1 \in \ZZ\).
    If \(\h A\) and \(\h B\) are \(m\)-interleaved in \(\Ho(\M)^\ZZ\), then
    \(A\) and \(B\) are \(2m\)-homotopy interleaved in \(\M^\ZZ\).
\end{prop}

\begin{proof}
    We start by giving the proof for the case \(m = 1\), as in this case the main idea is more clear.
    We will use the following constructions.
    Let \(\even \colon \ZZ \to \ZZ\) be the functor that maps even numbers to themselves and an odd number \(n\) to \(n-1\).
    Similarly, let \(\odd \colon \ZZ \to \ZZ\) be the functor that maps odd numbers to themselves
    and an even number \(n\) to \(n-1\).

    Note that, for every \(C \in \M^\ZZ\), we have that \(C\) is \((1,0)\)-interleaved with \(\even^*(C)\) and with
    \(\odd^*(C)\), and that \(\even^*(C)\) and \(\odd^*(C)\) are \(1\)-interleaved.

    Now assume given a \(1\)-interleaving between \(\h A\) and \(\h B\) in \(\Ho(\M)^\ZZ\), that is, assume that
    there are morphisms \(f_i \colon \h A(i) \to \h B(i+1)\) and \(g_i \colon \h B(i) \to \h A(i+1)\) in \(\Ho(\M)\) rendering
    the following diagram commutative:
\[
    \begin{tikzpicture}
      \matrix (m) [matrix of math nodes,row sep=5em,column sep=5em,minimum width=2em,nodes={text height=1.75ex,text depth=0.25ex}]
        { \cdots & \h A(-1) & \h A(0) & \h A(1) & \cdots \\
          \cdots & \h B(-1) & \h B(0) & \h B(1) & \cdots \\};
      \path[line width=0.75pt, -{>[width=8pt]}]
        (m-2-1) edge node [above] {$\,\,[\beta_{-2}]$} (m-2-2)
                edge [bend right=10] node [below left] {$g_{-2}\;\;\;\;\;\;\;\;$} (m-1-2)
        (m-2-2) edge node [above] {$[\beta_{-1}]$} (m-2-3)
                edge [bend right=10] node [below left] {$g_{-1}\;\;\;\;\;\;\;\;$} (m-1-3)
        (m-2-3) edge node [above] {$[\beta_0]$} (m-2-4)
                edge [bend right=10] node [below left] {$g_{0}\;\;\;\;\;\;\;\;$} (m-1-4)
        (m-2-4) edge node [above] {$[\beta_1]$} (m-2-5)
                edge [bend right=10] node [below left] {$g_{1}\;\;\;\;\;\;\;\;$} (m-1-5)
        (m-1-1) edge node [below] {$[\alpha_{-2}]$} (m-1-2)
                edge [bend left=10,-,line width=6pt,draw=white] (m-2-2)
                edge [bend left=10] node [above left] {$f_{-2}\;\;\;\;\;\;\;\;$} (m-2-2)
        (m-1-2) edge node [below] {$[\alpha_{-1}]$} (m-1-3)
                edge [bend left=10,-,line width=6pt,draw=white] (m-2-3)
                edge [bend left=10] node [above left] {$f_{-1}\;\;\;\;\;\;\;\;$} (m-2-3)
        (m-1-3) edge node [below] {$[\alpha_0]$} (m-1-4)
                edge [bend left=10,-,line width=6pt,draw=white] (m-2-4)
                edge [bend left=10] node [above left] {$f_{0}\;\;\;\;\;\;\;\;\;$} (m-2-4)
        (m-1-4) edge node [below] {$[\alpha_1]$} (m-1-5)
                edge [bend left=10,-,line width=6pt,draw=white] (m-2-5)
                edge [bend left=10] node [above left] {$f_{1}\;\;\;\;\;\;\;\;$} (m-2-5)
        ;
    \end{tikzpicture}
\]
    Consider the object \(C' \in \Ho(\M)^\ZZ\) given by one of the two diagonal
    zig-zags of the diagram above, namely, let
    \[
        C' =\;\;\; \cdots \xrightarrow{f_{-2}} \h B(-1) \xrightarrow{g_{-1}} \h A(0) \xrightarrow{f_0} \h B(1) \xrightarrow{g_1} \h A(2) \xrightarrow{f_2} \cdots
    \]
    Using \cref{h-is-smothering}, construct \(C \in \M^\ZZ\) such that \(\h C \cong C'\).

    Now, by construction, we have that \(\h(\even^*(A)) = \even^*(\h A) = \even^*(C') \cong \even^*(\h C) = \h(\even^*(C))\),
    so from \cref{h-is-smothering} it follows that \(\even^*(A) \simeq \even^*(C)\). Similarly, we have \(\odd^*(B) \simeq \odd^*(C)\).
    Since \(A\) is \((1,0)\)-interleaved with \(\even^*(A)\), \(\even^*(C)\) is \(1\)-interleaved with \(\odd^*(C)\),
    and \(\odd^*(B)\) is \((0,1)\)-interleaved with \(B\), \cref{triangle-inequality} implies that
    \(A\) and \(B\) are \(2\)-homotopy interleaved, concluding the proof for the case \(m = 1\).

    The proof for general \(m \geq 1 \in \ZZ\) is analogous, replacing the functor \(\even : \ZZ \to \ZZ\) with \(\even_m : \ZZ \to \ZZ\) given by \(\even_m(n) = \even(n\sslash m) \times m\),
    the functor \(\odd : \ZZ \to \ZZ\) with \(\odd_m : \ZZ \to \ZZ\) given by \(\odd_m(n) = \odd(n\sslash m) \times m\), and \(C' \in \Ho(\M)^\ZZ\) with
    \[
        C'(n) = \begin{cases}
            \h (\even_m^*(A))(n) & \text{ if \(n\sslash m\) is even }\\
            \h (\odd_m^*(B))(n)  & \text{ if \(n\sslash m\) is odd},
        \end{cases}
    \]
    where \(n \sslash m\) denotes the largest integer \(l\) such that \(l \times m \leq n\).
\end{proof}

We are now ready to prove the main result of this section.

\begingroup
\def\thetheorem{\cref{strictification-step}}
\addtocounter{thmx}{-2}
\begin{thmx}
    Let \(\M\) be a cofibrantly generated model category, let \(X, Y \in \M^\RR\),
    and let \(\delta > 0 \in \RR\).
    If \(X\) and \(Y\) are \(\delta\)-homotopy commutative interleaved, then
    they are \(c\delta\)-homotopy interleaved for every \(c > 2\).
\end{thmx}
\endgroup

\begin{proof}
    Let \(c > 2\) and let \(m \geq 1 \in \ZZ\) be large enough so that \((2m + 2)/m \leq c\).
    By \cref{delta-to-one}, we may assume that \(X,Y \in \M^\RR\) are \(m\)-homotopy commutative interleaved
    and that we must show that they are \(cm\)-homotopy interleaved.
    Since \(2m + 2 \leq mc\), \cref{from-R-to-Z} reduces the problem to showing that
    \(i^*(X)\) and \(i^*(Y)\) are \(2m\)-homotopy interleaved in \(\M^\ZZ\),
    knowing that they are $m$-homotopy commutative interleaved.
    \cref{noncoherent-to-strict-zed} now finishes the proof.
\end{proof}

\subsection{Lower bound}
\label{lower-bound-section}

\cref{strictification-step} implies that we have \(d_{HI} \leq c d_{HC}\) as distances on \(\M^\RR\), for \(c = 2\) and for every cofibrantly generated
model category \(\M\).
One could wonder if the constant \(c = 2\) can be improved.
In this section we show that, when \(\M = \SS\), we have \(c \geq 3/2\).
We do this by characterizing three-object persistent spaces which are \(1\)-homotopy interleaved with a trivial persistent space in
terms of the vanishing of a Toda bracket.
The idea of using Toda brackets to prove that \(d_{HI} \neq d_{HC}\) is suggested in \cite[Example~7.3]{lb}.

The Toda bracket is an operation on composable triples of homotopy classes of pointed maps,
and was originally defined to compute homotopy groups of spheres (\cite{Toda}).
We are interested in the use of Toda brackets as an algebraic obstruction to the rectification of diagrams.
We now describe the fundamental procedure involved in the definition
of Toda brackets, and the few properties that we are interested in
(see, e.g., \cite{BJT}).

Let \(\Sp\) denote the category of pointed spaces.
For concreteness, in the arguments of this section we use \(\SS = \Top\).
Let \([3]\) denote the category freely generated by the graph \(\bullet \to \bullet \to \bullet \to \bullet\).
A diagram \(X \in \Ho(\Sp)^{[3]}\), which is given by 
\(X(0),X(1),X(2),X(3) \in \Ho(\Sp)\) and homotopy classes of pointed maps
\([f_0] \colon X(0) \to X(1)\), \([f_1] \colon X(1) \to X(2)\), and \([f_2] \colon X(2) \to X(3)\),
is a \define{bracket sequence} if \([f_1] \circ [f_0]\) and \([f_2] \circ [f_1]\)
are equal to the null map, that is, to the homotopy class of the constant pointed map.

Let \(X' \in \Ho(\Sp)^{[3]}\) be a bracket sequence and let \(X \in \Sp^{[3]}\) be such
that \(\h X \cong X'\), which exists by \cref{h-is-smothering}.
We can, and do, assume that \(X\) takes values in CW-complexes.
Consider the following diagram of pointed spaces and pointed maps:
\[
    \begin{tikzpicture}
      \matrix (m) [matrix of math nodes,row sep=2em,column sep=1.5em,minimum width=2em,nodes={text height=1.75ex,text depth=0.25ex}]
        { X(0) & X(1) & \ast \\
          \ast & X(2) & X(3). \\};
      \path[line width=0.75pt, -{>[width=8pt]}]
        (m-1-1) edge node [above] {$f_0$} (m-1-2)
                edge node [left] {} (m-2-1)
        (m-2-1) edge (m-2-2)
        (m-1-2) edge node [right] {$f_1$} (m-2-2)
                edge node [left] {} (m-1-3)
        (m-2-2) edge node [above] {$f_2$} (m-2-3)
        (m-1-3) edge node [right] {} (m-2-3)
        ;
    \end{tikzpicture}
\]
Since \(X'\) is a bracket sequence, we know that there exist (pointed) homotopies filling
the squares in the diagram above.
For \(Y\) a pointed space, let \(C Y\) denote its reduced cone.
Each pair of such homotopies gives us pointed maps \(\alpha \colon CX(0) \to X(2)\) and \(\beta \colon CX(1) \to X(3)\)
such that \(\alpha \circ i = f_1 \circ f_0 \colon X(0) \to X(2)\) and \(\beta \circ i = f_2 \circ f_1\),
where \(i\) is the inclusion into the cone.
In particular, we have a commutative square
\begin{equation}\label{square-inducing-toda}
    \begin{tikzpicture}[baseline=(current  bounding  box.center)]
      \matrix (m) [matrix of math nodes,row sep=2em,column sep=2em,minimum width=2em,nodes={text height=1.75ex,text depth=0.25ex}]
        { X(0) & CX(1)\\
          CX(0)& X(3), \\};
      \path[line width=0.75pt, -{>[width=8pt]}]
        (m-1-1) edge node [above] {$i \circ f_0$} (m-1-2)
                edge node [left] {$i$} (m-2-1)
        (m-2-1) edge node [above] {$f_2 \circ \alpha$} (m-2-2)
        (m-1-2) edge node [right] {$\beta$} (m-2-2)
        ;
    \end{tikzpicture}
\end{equation}
which, by noticing that the pushout of the top and left morphisms is a model for the reduced
suspension of \(X(0)\), gives us an element of \([\Sigma X'(0), X'(3)]\),
where \([-,-]\) denotes homotopy classes of pointed maps.

\begin{defn}
    Let \(X \in \Ho(\Sp)^{[3]}\) be a bracket sequence.
    Consider the subset of \([\Sigma X(0), X(3)]\) consisting of
    all elements that can be obtained using
    the procedure above.
    This is the \define{Toda bracket} of \(X\).
    We say that the Toda bracket \define{vanishes} if it contains the null map.
\end{defn}

It is well-known (see, e.g., \cite[Section~1]{BJT}) that the non-vanishing of a Toda bracket is an obstruction
to the rectification of the bracket sequence, in the following sense.

\begin{prop}\label{toda-bracket-main-thm}
    The Toda bracket of a bracket sequence \(X' \in \Ho(\Sp)^{[3]}\)
    vanishes if and only if there exists \(X \in \Sp^{[3]}\)
    with \(\h X \cong X'\) and with \(f_1 \circ f_0\) and \(f_2 \circ f_1\) equal to the null map.\qed
\end{prop}

Although Toda brackets are defined for diagrams of pointed spaces, one can extend them
to unpointed spaces, provided the spaces are simply connected.
This is what we do now.
A \define{simply connected space} is a non-empty, connected space whose fundamental groupoid is trivial.
Let \(\Ss\) and \(\Sps\) denote the categories of simply connected spaces and of pointed, simply connected spaces, respectively.
We have the following well-known fact and corollary.
\begin{lem}
    \label{forget-point-equivalence}
    The forgetful functor \(U \colon \Ho(\Sps) \to \Ho(\Ss)\) is an equivalence of categories.\qed
\end{lem}

\begin{cor}
    \label{toda-bracket-well-defined}
    If \(X \in \Ho(\Sps)^{[3]}\) is such that the composite of consecutive maps of \(U_*(X)\) are
    null-homotopic, then \(X\) is a bracket sequence.
    
    Let \(X,X' \in \Ho(\Sps)^{[3]}\) be such that \(U_*(X) \cong U_*(X')\).
    Then \(X\) is a bracket sequence if and only if \(X'\) is;
    in that case, the Toda bracket of \(X\) vanishes if and only if the Toda bracket of \(X'\) does.
    \qed
\end{cor}

    \cref{toda-bracket-well-defined} implies that, for \(X \in \Ho(\Ss)^{[3]}\), there is a well-defined notion of
\(X\) being a bracket sequence, namely that any lift \(X' \in \Ho(\Sps)^{[3]}\) is a bracket sequence;
in that case, we say that the Toda bracket of \(X\) vanishes if the Toda bracket of \(X'\) does.

Let \(\j \colon \SS^{[3]} \to \SS^\ZZ\) be given by extending \(X \in \SS^{[3]}\) to the right with the singleton space
and to the left with the empty space.
Let \(\ast \in \SS^{[3]}\) be the constant singleton space.

\begin{prop}\label{characterization-homotopy-interleaved-with-ast}
    Let \(X \in \Ss^{[3]}\), then
    \begin{enumerate}
        \item \label{condition1-hi-trivial} \(\h(\j(X)) \in \Ho(\SS)^{\ZZ}\) is \(1\)-interleaved with \(\h(\j(\ast))\)
            if and only if \(\h X\) is a bracket sequence;
        \item \label{condition2-hi-trivial} if \(\h(\j(X)) \in \Ho(\SS)^{\ZZ}\) is \(1\)-interleaved with \(\h(\j(\ast))\),
            then \(\j(X)\) is \(1\)-homotopy interleaved with \(\j(\ast)\) if and only if
            the Toda bracket of \(\h X\) vanishes.
    \end{enumerate}
\end{prop}
\begin{proof}
    Statement \ref{condition1-hi-trivial} follows directly from \cref{toda-bracket-well-defined}.
    For \ref{condition2-hi-trivial}, note that if the Toda bracket of \(\h X\) vanishes, then, by \cref{toda-bracket-main-thm},
    there exists \(X' \in \Sps^{[3]}\) such that \(\h X' \cong \h X\) and such that the composite of consecutive
    maps of \(X'\) are null maps. In particular, \(\j(X')\) is \(1\)-interleaved with \(\j(\ast)\),
    and, since \(\h(\j(X')) \cong \h(\j(X))\), we have that \(\j(X') \simeq \j(X)\) by \cref{h-is-smothering},
    so \(\j(X)\) and \(\j(\ast)\) are \(1\)-homotopy interleaved.

    For the converse of \ref{condition2-hi-trivial}, assume that \(\j(X)\) and \(\j(\ast)\) are \(1\)-homotopy interleaved.
    It follows that there exists a commutative diagram of pointed spaces and pointed maps
\[
    \begin{tikzpicture}
      \matrix (m) [matrix of math nodes,row sep=2em,column sep=2em,minimum width=2em,nodes={text height=1.75ex,text depth=0.25ex}]
        { X'(0) & X'(1) & B\\
          A & X'(2) & X'(3), \\};
      \path[line width=0.75pt, -{>[width=8pt]}]
        (m-1-1) edge node [below] {$f'_0$} (m-1-2)
                edge node [left] {} (m-2-1)
        (m-1-2) edge node [right] {$f'_1$} (m-2-2)
                edge node [left] {} (m-1-3)
        (m-2-2) edge node [above] {$f'_2$} (m-2-3)
        (m-1-3) edge node [right] {} (m-2-3)
        (m-2-1) edge node [below] {} (m-2-2)
                edge [bend left=8,-,line width=6pt,draw=white,shorten >=8pt] (m-1-3)
                edge [bend left=8,shorten >=8pt] node [left] {} (m-1-3)
        ;
    \end{tikzpicture}
\]
    with \(A\) and \(B\) contractible and \(X' \in \Sp^{[3]}\) such that \(X' \simeq X\), as diagrams of unpointed spaces.
    It suffices to show that the Toda bracket of \(X'\) vanishes.
    For this, note that, using the diagonal morphism \(A \to B\),
    we can find maps \(\alpha \colon CX'(0) \to X'(2)\) and \(\beta \colon CX'(1) \to X'(3)\)
    such that \(\beta \circ Cf'_1 = f'_2 \circ \alpha\).
    In particular, in this case, there is a diagonal filler for the square \eqref{square-inducing-toda}
    and thus the induced map \(\Sigma X'(0) \to X'(3)\) is null-homotopic, as required.
\end{proof}

The following lemma is clear.

\begin{lem}
    \label{iterleaved-in-R-and-in-Z}
    Let \(X,Y \in C^\ZZ\).
    If \(i_*(X), i_*(Y) \in C^\RR\) are \(r\)-interleaved for some \(0 \leq r < 3/2\), then
    \(X,Y \in C^\ZZ\) are \(1\)-interleaved.\qed
\end{lem}

We are now ready to prove the lower bound.

\begin{prop}
    \label{lower-bound-prop}
    Let \(\M = \SS\). If \(d_{HI} \leq c d_{HC}\) then \(c \geq 3/2\).
\end{prop}
\begin{proof}
    By \cref{iterleaved-in-R-and-in-Z} and \cref{characterization-homotopy-interleaved-with-ast},
    it suffices to find a bracket sequence \(X \in \Ho(\Sp)^{[3]}\) valued in simply connected spaces
    such that its Toda bracket does not vanish.
    Examples of this are given in \cite{Toda}.
    A classical example, referenced in \cite{lb}, is
    \(S^4 \to S^4 \to S^3 \to S^3\) with the first and last maps
    degree \(2\) maps, and the middle map the suspension of the Hopf map.
\end{proof}

\begin{rmk}
    \label{discussion-distances-different}
    \cref{lower-bound-prop} implies in particular that \(d_{HI} \neq d_{HC}\).
    As mentioned in the introduction, we know that we have \(d_{HI} \geq d_{IHC} \geq d_{HC}\),
    so it is natural to wonder whether we have \(d_{HI} \neq d_{IHC}\) or \(d_{IHC} \neq d_{HC}\), or both.
    We leave these as open questions.
\end{rmk}

\subsection{Impossibility of rectification in higher dimensions}
\label{impossibility}

In this section, we show that \cref{strictification-step} has no analogue for multi-persistent spaces;
we thank Alex Rolle for pointing this out to us.
We prove this for $m=2$ and remark that a similar argument works for \(m > 2\).

\begin{prop}
    \label{impossibility-result}
If \(m = 2\), there is no constant \(c > 0 \in \RR\)
such that for all \(\delta > 0 \in \RR^m\),
if \(X,Y \in \SS^{\RR^m}\) are \(\delta\)-homotopy commutative interleaved,
    then they are \(c\delta\)-homotopy interleaved.
\end{prop}

Let \(\SQ\) denote the subposet of \(\RR^2\) spanned by \(\{(0,0), (0,1), (1,0), (1,1)\}\),
so that a functor \(\SQ \to C\) from \(\SQ\) to a category \(C\) corresponds
to a commutative square in \(C\).
We will use the following well-known fact, which says that
a homotopy commutative diagram can have different, non-equivalent lifts.
For a specific instance see, e.g., \cite{GW}.

\begin{lem}
    \label{impossibility-lemma}
There exist \(A,B \colon \SQ \to \SS\) such that
\(\h A \cong \h B \in \Ho(\SS)^\SQ\) and such that \(A \not\simeq B\).\qed
\end{lem}

\begin{proof}[Proof of \cref{impossibility-result}]
    Given a diagram \(A \colon \SQ \to \SS\), consider the bi-persistent space
    \(A' \colon \RR^2 \to \SS\) such that \(A'(r,s) = \emptyset\) whenever \(r\) or \(s\)
    are negative, \(A'(r,s) = A(\floor{r},\floor{s})\) whenever \(0\leq r,s < 2\),
    and \(A'(r,s)\) is the singleton space whenever \(0\leq r,s\) and \(2 \leq \max(r,s)\).
    Let \(A,B \colon \SQ \to \SS\).
    Note that if \((0,0) \leq \delta < (1/2,1/2) \in \RR^2\) and \(A',B' \in \SS^{\RR^2}\)
    are \(\delta\)-homotopy interleaved, then we have \(A \simeq B\).

    To prove the result, it is enough to show that
    there exist bi-persistent spaces \(X,Y \in \SS^{\RR^2}\) that are
    \(0\)-homotopy commutative interleaved, i.e., such that \(\h X \cong \h Y\),
    which are not \(\delta\)-homotopy interleaved for any \(0 \leq \delta < (1/2,1/2) \in \RR^2\).
    In order to do this, we can let \(A\) and \(B\) be as in \cref{impossibility-lemma} and
    take \(X = A'\) and \(Y = B'\).
\end{proof}

\section{Projective cofibrant persistent simplicial sets}
\label{cofibrant-spaces-section}

The purpose of this section is to characterize projective cofibrant persistent simplicial sets as filtered simplicial sets (\cref{t: cofibrant = filtered}).
We work with simplicial sets indexed by an arbitrary poset $(P,\leq)$.

\begin{defn}
   A \define{\(P\)-filtered simplicial set} (\define{filtered simplicial set} when there is no risk of confusion) is a simplicial set \(X\) equipped with functions \(\beta_n \colon X_n \to P\), satisfying:
   \begin{itemize}
    \item \(\beta_{n-1}(d_i(\sigma)) \leq \beta_n(\sigma)\)
    for every \(n \geq 1\), \(\sigma \in X_n\), and boundary map \(d_i \colon X_n \to X_{n-1}\);
    \item \(\beta_{n+1}(s_i(\sigma)) \leq \beta_n(\sigma)\) for every \(n \geq 0\),
    \(\sigma \in X_n\), and degeneracy map \(s_i \colon X_n \to X_{n+1}\).
\end{itemize}
\end{defn}
When there is no risk of ambiguity, we denote the filtered simplicial set \((X,\beta)\) simply by \(X\).

\begin{defn}
Given a filtered simplicial set \((X,\beta)\) define a persistent simplicial set \(\widehat{(X,\beta)} \in \sSet^{P}\) such that
for \(r  \in P\) we have $\widehat{(X,\beta)}(r)_n = \left\{\sigma \in X_n : \beta_n(\sigma) \leq r\right\}$,
with faces and degeneracies given by restricting the ones of \(X\).
\end{defn}

By a standard abuse of language,
We say that a persistent simplicial set \define{is a filtered simplicial set}
if it is isomorphic to \(\widehat{Y}\), for \(Y\) a filtered simplicial set.

The following result is a characterization of filtered simplicial sets among persistent simplicial sets by means of easily verified point-set conditions.

\begin{lem}
\label{p: filtered char}
A persistent simplicial set \(X \in \sSet^{P}\) is a filtered simplicial set if and only if the following conditions are satisfied:
\begin{enumerate}
	\item \label{cond-1} The structure morphism \(X(r) \to X(r')\) is a monomorphism for every \(r\leq r'\) in \(P\).
        In particular, up to isomorphism, we may, and do, assume that \(X(r)\) is a subsimplicial set of \(X(r')\).
    \item \label{cond-2} For every simplex \(\sigma\in \bigcup_{r \in P}X(r)\), the set \(\{t \in P \ \colon \ \sigma \in X(t)\}\) has a minimum.
\end{enumerate}
\end{lem}
\begin{proof}
    The only non-trivial part is that if \(X\) satisfies the two conditions in the statement then it is filtered. Set \(Y = \bigcup_{r \in P} X(r)\), which makes sense thanks to condition \ref{cond-1}.
Given \(\sigma \in Y_n\), define \(\beta_n(\sigma) := \min\{t \in P \ \colon \ \sigma \in X(t)\}\), which is well-defined thanks to condition \ref{cond-2}.
We then have $X \cong \widehat{Y}$.
The rest of the proof is clear.
\end{proof}

The proof of the following lemma a straightforward application of Lemma \ref{p: filtered char}.
We use the term \define{cell attachment} to indicate any pushout of a generating cofibration \(r\odot \partial \Delta^n \to r\odot \Delta^n\).
\begin{lem}
    \label{retract + pushout + filtcolim-of-filtered-is-filtered}
    \begin{enumerate}
    \item A retract of a filtered simplicial set is filtered.
    \item If the domain of a cell attachment is a filtered simplicial set, then the codomain is too.
\item Let \(\zeta\) be a limit ordinal and let \(X_{\bullet}\colon\zeta \to \sSet^{P}\) be a diagram of persistent simplicial sets, where for each \(\gamma < \zeta \) we have that the map \(X_{\gamma} \to X_{\gamma+1}\) is a cell attachment. If \(X_{\gamma}\)
is a filtered simplicial set for every \(\gamma < \zeta \), then \(X_{\zeta } = \colim_{\gamma < \zeta }X_{\gamma}\) is a filtered simplicial set.\qed
    \end{enumerate}
\end{lem}

The recognition principle for projective cofibrant persistent simplicial sets is now a consequence of \cref{retract + pushout + filtcolim-of-filtered-is-filtered} and the fact that the cofibrant objects in a cofibrantly generated model category are precisely the retracts of transfinite compositions of cell attachments (\cite[Proposition~2.1.18(b)]{Ho}).

\begin{prop}
\label{t: cofibrant = filtered}
    A persistent simplicial set is filtered if and only if it is projective cofibrant.\qed
\end{prop}

In practice, many of the persistent spaces relevant to Topological Data Analysis are filtered
simplicial sets.

\begin{eg}
\label{VR-complex example}
The Vietoris--Rips complex associated to a metric space \((X,d_X)\), usually defined to be a persistent simplicial complex, can be turned into a persistent simplicial set by choosing a total order on $X$.
It follows directly from its definition that this persistent simplicial set is filtered.
Other examples of this form include the \v{C}ech complex and the filtrations of \cite{MC}.

An example of a filtered multi-persistent simplicial set is the following.
Given a metric space \((X,d_X)\) together with a real-valued function \(f : X \to \RR\),
one can construct a bi-filtered simplicial set as follows.
For each \(s \in \RR\), consider \(X_s = f^{-1}(-\infty,s]\) and let \(F_{s,r}\) be the Vietoris--Rips complex of $X_s$ at scale $r$.
\end{eg}

We remark that persistent simplicial sets whose structure maps are monomorphism are not necessarily filtered.
This happens in practice when the same simplex ``appears at different times'', that is, when condition \ref{cond-2} in \cref{p: filtered char} is not satisfied.
Examples of this include the \textit{degree-Rips} bi-filtration (\cite{LW}), and Vietoris--Rips applied to the \textit{kernel density filtration} of \cite{rolle2020stable}.

\section{Interleaving in \(\Ho\left(\SS^{\RR^m}\right)\) and in homotopy groups}
\label{homotopy-category-and-homotopy-groups}

In this section, we prove \cref{homotopy-groups-to-noncoherent}.
We start by defining the notions of persistent homotopy groups of a persistent space, and of
morphism inducing an interleaving in persistent homotopy groups.
The notion of persistent homotopy group we use is essentially
the same as that of Jardine (\cite{Jardine2020}).

We model the \(n\)th homotopy group \(\pi_n(W,w)\) of a pointed space \((W,w)\) by the set of pointed homotopy classes of pointed maps from the \(n\)th dimensional sphere \(S^n\) into \(W\).

\begin{defn}
    Let \(X \in \pSm\).
    The persistent set \(\pi_0(X) \colon \RR^m \to \Set\) is defined by \(\pi_0 \circ X\).
    Let \(n \geq 1\), \(r \in \RR^m\), and \(x \in X(r)\).
    The \define{\(n\)th persistent homotopy group} of \(X\) based at \(x\) is the persistent
    group \(\pi_n(X,x) \colon \RR^m \to \Grp\) that is trivial at \(s \ngeq r\),
    and that is \(\pi_n(X(s),\phi^X_{r,s}(x)) \in \Grp\) at \(s \geq r\).
 \end{defn}

Note that \(\pi_n\) is functorial for every \(n \in \N\).

\begin{defn}
\label{interleaving-in-homotopy-groups}
Let \(\epsilon,\delta \geq 0 \in \RR^{m}\).
Assume given a homotopy class of morphisms \([f] \colon X' \to Y'^\epsilon \in \Ho\left(\SS^{\RR^m}\right)\).
Let \(X' \simeq X\) be a cofibrant replacement, let \(Y' \simeq Y\) be a fibrant replacement,
and let \(f \colon X \to_\epsilon Y\) be a representative of \(f\).
We say that \([f]\) \define{induces an \((\epsilon,\delta)\)-interleaving in homotopy groups} if the induced map \(\pi_0(f) \colon \pi_0(X) \to_\epsilon \pi_0(Y)\) is part of an \((\epsilon,\delta)\)-interleaving
of persistent sets, and if for every \(r \in \RR^m\), every \(x \in X(r)\), and every \(n \geq 1 \in \N\),
the induced map
\(\pi_n(f) : \pi_n(X,x) \to_\epsilon \pi_n(Y,f(x))\) is part of an \((\epsilon,\delta)\)-interleaving
of persistent groups.
\end{defn}

It is clear that the definition above is independent of the choices of representatives.

A standard result in classical homotopy theory is that a fibration of Kan complexes
inducing an isomorphism in all homotopy groups has the right lifting property with respect
to cofibrations (\cite[Theorem~I.7.10]{GJ}).
An analogous, persistent, result
(\cref{lifting-property-n-cofibrations}), says that, for a fibration of fibrant objects
inducing a \(\delta\)-interleaving in homotopy groups, the lift exists up to a shift, which depends on both
\(\delta\) and on a certain ``length'' \(n \in \N\) associated to the cofibration.
To make this precise, we introduce the notion of \(n\)-dimensional extension.

\begin{defn}
\label{def n cofibration}
    Let \(A,B \in \pSm\) and let \(n \in \N\).
    A map \(j\colon A \to B\) is a \define{\(n\)-dimensional extension} (of \(A\)) if there exists a set \(I\), a family of tuples of real numbers
    \(\left\{r_i \in \RR^m\right\}_{i \in I}\), and commutative squares of the form depicted on the left below, that together give rise to the pushout square on the right below.
Here, \(\partial D^n \hookrightarrow D^n\) stands for \(S^{n-1}\hookrightarrow D^n\) if \(\ \SS = \Top\), and for \(\partial \Delta^n \hookrightarrow \Delta^n\) if  \(\ \SS = \sSet\).
    \[\begin{tikzcd}
     \partial D^n \ar[r,"f_i"] \ar[d,hook] &  A(r_i) \ar[d,"j_{r_i}"] & & & & \coprod_{i \in I} r_i \odot(\partial D^n) \ar[r,"f"] \ar[d] &  A \ar[d,"j"]\\
     D^n \ar[r,"g_i"] & B(r_i) & & & &  \coprod_i r_i \odot( D^n) \ar[r,"g"] & B
    \end{tikzcd}\]  
\end{defn}

A \define{single dimensional extension} is an \(n\)-dimensional extension for some \(n\in \N\).

\begin{defn}
Let \(\iota \colon A \to B\) be a projective cofibration of \(\pSm\) and let \(n \in \N\).
We say that \(\iota\) is an \define{\(n\)-cofibration} if it factors as the composite
of \(n+1\) maps \(f_0, \dots, f_n\), with \(f_i\) an \(n_i\)-dimensional extension for some \(n_i \in \N\).
We say that \(A\in \pSm\) is \(n\)-cofibrant if the map \(\emptyset \to A\) is an \(n\)-cofibration.
\end{defn}

The next lemma, which follows directly from \cref{t: cofibrant = filtered}, gives a rich family of examples of \(n\)-cofibrant persistent simplicial sets.
Recall that a simplicial set is \(n\)-skeletal if all its simplices in dimensions above \(n\) are degenerate.

\begin{lem}
    \label{sset-n-cofibrant}
    Let \(A \in \sSet^{\RR^m}\) and let \(n \in \N\).
    If \(A\) is projective cofibrant and pointwise \(n\)-skeletal, then it is \(n\)-cofibrant.\qed
\end{lem}

\begin{eg}
The Vietoris--Rips complex \(\VR(X)\) of a metric space \(X\), as defined in \cref{VR-complex example}, is \(n\)-cofibrant if the underlying set of \(X\) has finite cardinality \(\vert X \vert = n + 1\).

If one is interested in persistent (co)homology of some bounded degree \(n\), then one can restrict computations to the \((n+1)\)-skeleton of a Vietoris--Rips complex,
which is \((n+1)\)-cofibrant.
\end{eg}

A result analogous to \cref{sset-n-cofibrant}, but for persistent topological spaces, does not hold, as cells are not necessarily attached in order of dimension.
This motivates the following definition.

\begin{defn}
\label{persistent CW}
    Let \(n\in \N\).
    A persistent topological space \(A \in \Top^{\RR^m}\) is an
    \define{\(n\)-dimensional persistent CW-complex} if
    the map \(\emptyset \to X\) can be factored as a composite
    of maps \(f_0, \dots, f_n\), with \(f_i\) an \(i\)-dimensional extension.
\end{defn}

\begin{eg}
    The geometric realization of any \(n\)-cofibrant
    persistent simplicial set is an \(n\)-dimensional persistent CW-complex.
\end{eg}

\begin{lem}
    \label{cw-n-cofibrant}
    Every \(n\)-dimensional persistent CW-complex is \(n\)-cofibrant.\qed
\end{lem}

We now make precise the notion of lifting property up to a shift.

\begin{defn}
    Let \(i \colon A \to B\) and \(p \colon Y \to X\) be morphisms in \(\pSm\) and let \(\delta \geq 0\).
    We say that \(p\) has the \define{right \(\delta\)-lifting property} with respect to \(i\)
    if for all morphisms \(A \to Y\) and \(B \to X\) making the square on the left below commute, there exists
    a diagonal \(\delta\)-morphism \(f \colon B \to_\delta Y\) rendering the diagram commutative.
    Below, the diagram on the left is shorthand for the one on the right.
    \[
    \begin{tikzpicture}
      \matrix (m) [matrix of math nodes,row sep=3em,column sep=3em,minimum width=2em,nodes={text height=1.75ex,text depth=0.25ex}]
        { A & Y  & & & A & Y & Y^\delta\\
          B & X & & & B & X & X^\delta.\\};
      \path[line width=0.75pt, -{>[width=8pt]}]
        (m-1-1) edge node [above] {} (m-1-2)
                edge node [left] {$i$} (m-2-1)
        (m-2-1) edge [dotted] node [above] {} node [at end, below] {$\leng{\delta}\;\;\,$} (m-1-2)
                edge node [above] {} (m-2-2)
        (m-1-2) edge node [right] {$p$} (m-2-2)

        (m-1-5) edge node [above] {} (m-1-6)
                edge node [left] {$i$} (m-2-5)
         (m-1-6) edge node [right] {$p$} (m-2-6)
                edge node [above] {$\shi_{0,\delta}(\id_Y)$} (m-1-7)
        (m-1-7) edge node [right] {$p^\delta$} (m-2-7)
        (m-2-5) edge [bend left=10,-,line width=6pt,draw=white] (m-1-7)
                 edge [bend left=10, dotted] node [left] {$\,\,\,\,\,\,$} (m-1-7)
                 edge node [above] {} (m-2-6)
        (m-2-6) edge node [below] {$\shi_{0,\delta}(\id_X)$} (m-2-7)
        ;
        
    \end{tikzpicture}\]

\end{defn}

We now prove \cref{jardines-lemma}, an adaptation of a result of Jardine, which says that fibrations inducing interleavings in homotopy groups have a shifted right lifting property, as defined above.
The main difference is that we work in the multi-persistent setting.
We use simplicial notation and observe that the corresponding statement for persistent topological spaces follows from the simplicial one by using the singular complex-realization adjunction.
We recall a standard, technical lemma whose proof is given within that of, e.g., \cite[Theorem~I.7.10]{GJ}.
\begin{lem}
\label{homotopy-lifting}
Suppose given a commutative square of simplicial sets
\begin{equation}
\label{lp2}
\begin{tikzcd}
\partial \Delta^n \ar[r,"\alpha"] \ar[d,hook] & X \ar[d,"p"]\\
\Delta^n \ar[r,"\beta"] &Y,
\end{tikzcd}
\end{equation}	
where \(p\) is a Kan fibration between Kan complexes. If there is commutative diagram like the one on the left below, for which the lifting problem on the right admits a solution, then the initial square \eqref{lp2} admits a solution.
\[\begin{tikzcd}
\partial \Delta^n \ar[ddd,hook] \ar[dr,"(\id_{\partial \Delta^n} \times \{1\}) "] \ar[drrr,"\alpha", bend left]&&\\
& \partial \Delta^n \times \Delta^1	\ar[d,hook]	\ar[rr," h"] & & X \ar[d,"p"] & & & \partial \Delta^n \ar[rr,"h \circ (\id_{\partial \Delta^n} \times \{0\}) "] \ar[d,hook] & & X \ar[d,"p"]\\
&\Delta^n \times \Delta^1 \ar[rr,"g"] & & Y & & & \Delta^n \ar[rr,"g \circ( \id_{\Delta^n} \times \{0\})"] &&Y\\
\Delta^n \ar[ur,"(\id_{\Delta^n} \times \{1\}) "{swap}] \ar[urrr,"\beta"{swap},bend right] &&
\end{tikzcd}\]
\end{lem}

\begin{lem}[cf.~{\cite[Lemma~14]{Jardine2020}}]
    \label{jardines-lemma}
    Let \(\delta \geq 0\), and let \(f \colon X \to Y \in \SS^{\RR^m}\) induce a \((0,\delta)\)-interleaving in homotopy groups.
    If \(X\) and \(Y\) are projective fibrant and \(f\) is a projective fibration, then \(f\) has the right
    \(2\delta\)-lifting property with respect to boundary inclusions \(r\odot \partial D^n \to r \odot D^n\),
    for every \(r \in \RR^m\) and every \(n \in \N\).
\end{lem}
\begin{proof}
Suppose given a commutative diagram as on the left below, which corresponds to the one on the right:
\begin{equation}
\label{lifting problem}
\begin{tikzcd}
r\odot \partial \Delta^n \ar[r,"a"] \ar[d] & X \ar[d,"p"] & & & \partial \Delta^n \ar[r,"\alpha"] \ar[d] & X(r) \ar[d,"p_r"]\\
r\odot  \Delta^n \ar[r,"b"] & Y & & & \Delta^n \ar[r,"\beta"] &Y(r).
\end{tikzcd}
\end{equation}
We must find a \(2\delta\)-lift for the diagram on the right.
The proof strategy is to appeal to \cref{homotopy-lifting} to simplify \(\alpha\), then prove that at the cost of a \(\delta\)-shift we can further reduce \(\alpha\) to a constant map, and then show that the simplified lifting problem can be solved at the cost of another \(\delta\)-shift. So we end up with a \(2\delta\)-lift, as in the statement. 
We proceed by proving the claims in opposite order.

We start by showing that 
\eqref{lifting problem} can be solved up to a \(\delta\)-shift whenever \(\alpha\) is constant.
Let us assume that \(\alpha\) is of the form \(\alpha = \ast\) for some \(\ast \in X(r)_0\). Since, then, \(\beta\) represents an element \([ \beta] \in \pi_n(Y(r),\ast)\), there exists a map \(\alpha' \colon \Delta^n \to X(r+\delta)\) whose restriction to \(\partial \Delta^n\) is constant on \(\ast \in X(r)_0\), and such that there is a homotopy \(h\colon \beta   \simeq p\alpha'\) relative to \(\partial\Delta^n\). We can thus consider
\[\begin{tikzcd}
\partial \Delta^n \ar[r, "\id_{\partial \Delta^n}\times\{1\}"] \ar[d,"i"] & \left(\partial \Delta^n \times \Delta^1) \cup (\Delta^n \times \{0\}\right) \ar[d] \ar[r,"(\ast {,} \alpha')"] & X(r+\delta) \ar[d,"p_{r+\delta}"]\\
\Delta^n  \ar[r, "\id_{\Delta^n}\times\{1\}"] &\Delta^n \times \Delta^1 \ar[r,"h"] \ar[ur, "H",dotted] & Y(r+\delta),
\end{tikzcd}\]
where \(H\) is a diagonal filler for the right-hand side square, which exists since the middle vertical map is a trivial cofibration of simplicial sets and \(p_{r+\delta}\) is a Kan fibration by assumption. The composite map \(H\circ \id_{\Delta^n}\times\{1\}\) is a lift for \eqref{lifting problem}. 

We now assume that \(\alpha\) is of a specific, simplified form, and prove that, up to a \(\delta\)-shift, we can reduce the lifting problem \eqref{lifting problem} to the case in which \(\alpha\) is constant.
Let us assume that \(d_i(\alpha) = \ast \in X(r)_0\) for every \(0<i\leq n\), and set \(\alpha_{0} = d_0(\alpha)\). We have that \(\alpha_0\) represents an element \([\alpha_0] \in \pi_{n-1}(X(r),\ast)\), with the property that \(p[\alpha_0] = 0  \in \pi_{n-1}(Y(r),\ast)\). Since \(p\) induces a \((0,\delta)\)-interleaving in homotopy groups, we have that \(\phi^X_{r,r+\delta}([\alpha_0])=0 \in \pi_{n-1}(X(r + \delta),\ast)\), witnessed by a homotopy \(h_0 \colon \Delta^{n-1} \times \Delta^1 \to X(r+\delta)\), constant on \(\partial \Delta^{n-1}\). If we set \(h'_i = \ast \colon \Delta^{n-1} \to X(r+\delta)\) for every \(0<i\leq n-1\) and \(h'_0 = h_0\), we get a map \(h'\colon\partial \Delta^n \times \Delta^1 \to X(r + \delta)\). We can now extend \((\phi^Y_{r,r+\delta}\circ \beta{,}ph') \colon ( \Delta^n \times \{1\}) \cup (\partial \Delta^n \times \Delta^1) \to  Y(r + \delta) \) to a homotopy \(H'\colon \Delta^n \times \Delta^1 \to Y(r + \delta) \).
Now observe that the following lifting problem is such that \(h' \circ (\id_{\partial \Delta^n} \times \{0\} )= *\), so, thanks to \cref{homotopy-lifting},
we have reduced this case to the case in which \(\alpha\) is constant.
\[\begin{tikzcd}
\partial \Delta^n \ar[rr,"h'_1 \circ (\id_{\partial \Delta^n} \times \{0\})"] \ar[d] & & X(r+\delta) \ar[d,"p_{r+\delta}"]\\
\Delta^n \ar[rr,"H' \circ (\id_{\Delta^n} \times \{0\}"] && Y(r + \delta).	
\end{tikzcd}\]

To conclude, we must show that we can reduce the original lifting problem \eqref{lifting problem} to one in which all but the \(0\)th faces of \(\alpha\) are constant on a point \(\ast \in X(r)_0\). Let \(K\colon \Lambda^n_0 \times \Delta^1 \to \Lambda^n_0\) be the homotopy that contracts the simplicial horn onto its vertex \(0\), which determines a diagram
\[\begin{tikzcd}
    \Lambda^n_0 \ar[r,"\id_{\Lambda^n_0}\times \{1\}"] \ar[d,hook] & \Lambda^n_0 \times \Delta^1 \ar[d,"k_1"] & \Lambda^n_0 \ar[l,"\id_{\Lambda^n_0}\times \{0\}"{swap}] \ar[d]\\
    \partial \Delta^n \ar[r,"\alpha"] & X & \Delta^0 \ar[l,"\alpha(0)"{swap}]
\end{tikzcd}\]
with \(k_1 = \alpha\circ j \circ K\), with \(j \colon \Lambda_0^n \to \partial \Delta^n\) the inclusion of the horn into the boundary. We can now extend the map \((\alpha{,}k_1) \colon \left(\partial \Delta^n \times \{1\}\right) \cup \left(\Lambda^n_0 \times \Delta^1\right) \to X(r)\) to a homotopy \(k\colon \partial \Delta^n \times \Delta^1  \to X(r) \). Similarly, we extend the map \((\beta{,}p \circ k) \colon \left( \Delta^n \times \{1\}\right) \cup \left(\partial \Delta^n \times \Delta^1\right) \to Y(r)\) to a homotopy \(g\colon \Delta^n \times \Delta^1 \to Y(r)\).
It now suffices to consider the diagram
\[\begin{tikzcd}
\partial \Delta^n \ar[ddd,hook] \ar[dr,"(\id_{\partial \Delta^n} \times \{1\}) "] \ar[drrr,"\alpha", bend left]&&\\
& \partial \Delta^n \times \Delta^1 \ar[d,hook] \ar[rr," k"] & & X \ar[d,"p"] & \\
&\Delta^n \times \Delta^1 \ar[rr,"g"] & & Y \\
\Delta^n \ar[ur,"(\id_{\Delta^n} \times \{1\}) "{swap}] \ar[urrr,"\beta"{swap},bend right] &&
\end{tikzcd}\]
observing that \(\alpha' := k_{\vert \partial\Delta^n \times \{0\}}\) satisfies \(d_i(\alpha' )= \ast\) for \(0<i\leq n\), and appeal to \cref{homotopy-lifting}.
\end{proof}

\begin{cor}
    \label{lifting-property-n-cofibrations}
    Let \(\delta \geq 0\) and let \(f \colon X \to_\delta Y\) induce a \(\delta\)-interleaving in homotopy groups.
    If \(X\) and \(Y\) are projective fibrant and \(f\) is a projective fibration, then \(f\) has the right
    \((4(n+1)\delta)\)-lifting property with respect to \(n\)-cofibrations, for all \(n \in \N\).
\end{cor}
\begin{proof}
    By assumption, \(f \colon X \to Y^\delta\) induces a \((0,2\delta)\)-interleaving in all homotopy groups.
    Now, an \(n\)-cofibration can be written as a composite of \(n+1\) single dimensional extensions, and any shift
    of a single dimensional extension is again a single dimensional extension, so it is enough
    to show that \(f\) has the right \(4\delta\)-lifting property with respect to single dimensional extensions.

    A single dimensional extension is the pushout of a coproduct
    \(\coprod_{i \in I} r_i \odot(\partial D^n) \to \coprod_{i \in I} r_i \odot D^n\), so it is enough
    to show that \(f\) has the right \(4\delta\)-lifting property with respect to coproducts of that form,
    which follows from \cref{jardines-lemma} and the universal property of coproducts.
\end{proof}

We are ready to prove \cref{homotopy-groups-to-noncoherent}.

\begingroup
\def\thetheorem{\cref{homotopy-groups-to-noncoherent}}
\begin{thmx}
    Let \(X,Y \in \SS^{\RR^m}\) be persistent spaces that are assumed to be
    projective cofibrant and \(d\)-skeletal if \(\SS = \sSet\), or persistent CW-complexes
    of dimension at most \(d\) if \(\SS = \Top\).
    Let \(\delta \geq 0 \in \RR^m\).
    If there exists a morphism in the homotopy category \(X \to Y^\delta \in \Ho\left(\SS^{\RR^m}\right)\)
    that induces \(\delta\)-interleavings in all homotopy groups,
    then \(X\) and \(Y\) are \((4(d+1)\delta)\)-interleaved in the homotopy category.
\end{thmx}
\endgroup

\begin{proof}
    By \cref{sset-n-cofibrant} and \cref{cw-n-cofibrant}, $X$ and $Y$ are $d$-cofibrant.
    Let \([f] \colon X \to Y^\delta\) be as in the statement. Since \([f]\) is a morphism in the homotopy category,
    we begin by choosing a convenient representative of it.
    We let \(p \colon X' \to Y'\) be a projective fibration between projective fibrant objects
    such that there exist trivial cofibrations \(i\colon X \to X'\) and \(j \colon Y \to Y'\)
    with \([p] \circ [i] = [j] \circ [f]\), in \(\Ho\left(\SS^{\RR^m}\right)\).

    Note that \([p]\) induces a \((0,2\delta)\)-interleaving in homotopy groups, between \(X'\) and \(Y'^\delta\).
    Since \(Y^{\delta}\) is \(d\)-cofibrant, \cref{lifting-property-n-cofibrations} guarantees that we can find a \((4(d+1)\delta)\)-lift \(g'\) of \(p\) against \(\emptyset \to Y\). We can then use the fact that \(j : Y \to Y'\) is a trivial cofibration, and that \(X'\) is fibrant, to 
    construct the following lift
    \[
    \begin{tikzpicture}
      \matrix (m) [matrix of math nodes,row sep=1.5em,column sep=2em,minimum width=2em,nodes={text height=1.75ex,text depth=0.25ex}]
        { Y & & X'^{(4d + 3)\delta} \\
          Y' & & \\};
      \path[line width=0.75pt, -{>[width=8pt]}]
        (m-1-1) edge node [above] {$g'$} (m-1-3)
                edge node [left] {$j$} (m-2-1)
        (m-2-1) edge [dotted] node [below] {$g$} (m-1-3)
        ;
    \end{tikzpicture}
    \]
    We will show that $\shi_{\delta,4(d+1)\delta}([p]) \colon X' \to_{4(d+1)\delta} Y'$ and $\shi_{(4d+3)\delta, 4(d+1)\delta}([g]) \colon Y' \to_{4(d+1)\delta} X'$
    form a \((4(d+1)\delta)\)-interleaving in the homotopy category between \(X'\) and \(Y'\).

    On the one hand, note that, by construction, we have \(p^{(4d+3)\delta} \circ g \circ j = p^{(4d+3)\delta} \circ g' = j\),
    so, since \([j]\) is an isomorphism, it follows that \([p]^{(4d+3)\delta} \circ [g] = \shi_{4(d+1)\delta}([\id_{Y'}])\),
    and thus that
    \[
        \shi_{\delta,4(d+1)\delta}([p])^{4(d+1)\delta} \circ \shi_{(4d+3)\delta, 4(d+1)\delta}([g]) = \shi_{8(d+1)\delta}([\id_{Y'}]).
    \]

    On the other hand, since \(X\) is cofibrant and \(Y'\) is fibrant,
    it follows from the previous paragraph that \(p^{4(d+1)\delta} \circ g^\delta \circ p \circ i \colon X \to Y'^{(4d+5)\delta}\) is homotopic to \(p^{4(d+1)\delta}\circ \shi_{0,4(d+1)\delta}(i)\). Let \(H \colon I \times X \to Y'^{(4d+5)\delta}\) be a homotopy between these maps, which gives the following commutative diagram
\[
    \begin{tikzpicture}
      \matrix (m) [matrix of math nodes,row sep=3em,column sep=7em,minimum width=2em,nodes={text height=1.75ex,text depth=0.25ex}]
        { X \coprod X & & X'^{4(d+1)\delta} \\
          I \times X & & Y'^{(4d+5)\delta}, \\};
      \path[line width=0.75pt, -{>[width=8pt]}]
        (m-1-1) edge node [above] {$\left(\shi_{0,4(d+1)\delta}(i){,}g^\delta\circ p \circ i\right)$} (m-1-3)
                edge node [left] {$(i_0,i_1)$} (m-2-1)
        (m-2-1) edge node [above] {$H$} (m-2-3)
        (m-1-3) edge node [right] {$p^{4(d+1)\delta}$} (m-2-3)
        ;
    \end{tikzpicture}
    \]
    where the left vertical map is the inclusion of into the cylinder.
    We claim that, since \(X\) is \(d\)-cofibrant, the inclusion into the cylinder is a \(d\)-cofibration. Indeed,
    a cell decomposition of this map is obtained by attaching an \((n+1)\)-cell for each \(n\)-cell in the decomposition of \(X\).
    Now, by \cref{lifting-property-n-cofibrations}, we can find a
    \((4(d+1)\delta)\)-lift of the diagram, which shows that
    \[
    \shi_{4(d+1)\delta,8(n+1)\delta}([g]^\delta \circ [p] \circ [i])
    =
    \shi_{0,8(d+1)\delta}([i]) : X \to X'^{8(d+1)\delta}.
    \]
    Since the left hand side is equal to
    \(\shi_{(4d+3)\delta, 4(d+1)\delta}([g])^{4(d+1)\delta} \circ \shi_{\delta, 4(d+1)\delta}([p] \circ [i])\), and \([i]\) is an isomorphism, it follows that
    \([g]^{4(d+1)\delta} \circ \shi_{\delta, 4(d+1)\delta}([p]) =
    \shi_{8(d+1)\delta}([\id_{X'}])\).
\end{proof}

\begin{rmk}
\label{comparison-persistent-whitehead}
Together, \cref{strictification-step} and \cref{homotopy-groups-to-noncoherent} imply 
a version of the persistent Whitehead conjecture, which we recall as \cref{bl-conjecture}.
Our result is, in a sense, stronger than the
one conjectured, since \cref{homotopy-groups-to-noncoherent},
which addresses part \((i)\) of the conjecture, applies to arbitrary multi-persistent spaces.
In another respect, our result is slightly weaker,
as the conjecture is stated for cofibrant, pointwise CW-complexes,
which does not necessarily imply being a persistent CW-complex in our sense.
We believe that this is not an issue, as many of the cofibrant, pointwise CW-complexes
persistent topological spaces that appear in applications are in fact persistent CW-complexes,
as they are usually the geometric realization of a filtered simplicial complex.
\end{rmk}

\begin{conjecture}[{\cite[Conjecture~8.6]{lb}}]
    \label{bl-conjecture}
    Suppose we are given connected, cofibrant $X,Y : \RR \to \Top$, with each $X(r)$ and $Y(r)$ CW-complexes of dimension at most $d$, and $f : X \to Y^\delta$ inducing a $\delta$-interleaving in all homotopy groups.
    Then, there is a constant $c$, depending only on $d$, such that
    \begin{itemize}
        \item[(i)] $f$ induces a $c\delta$-interleaving in the homotopy category $\Ho\left(\Top^\RR\right)$;
        \item[(ii)] $X$ and $Y$ are $c\delta$-homotopy interleaved.
    \end{itemize}
\end{conjecture}

\Urlmuskip=0mu plus 1mu\relax
\printbibliography

\end{document}